\documentclass[11pt]{amsart}

\usepackage{amssymb,latexsym,amsfonts}
\usepackage[cmtip,all]{xy}
\usepackage{enumerate}
\usepackage{geometry}
\usepackage{mathrsfs}
\usepackage[latin1]{inputenc}
\usepackage[francais,english]{babel}
\usepackage{epic}

\usepackage{tikz}

\newtheorem{theorem}{Theorem}[section]
\newtheorem{theo}[theorem]{Theorem}
\newtheorem{lem}[theorem]{Lemma}

\newtheorem{prop}[theorem]{Proposition}
\newtheorem{cor}[theorem]{Corollary}
\newtheorem{crit}[theorem]{Criterion}
\theoremstyle{df}
\newtheorem{df}[theorem]{Definition}
\theoremstyle{notation}

\theoremstyle{fact}

\theoremstyle{definition}

\theoremstyle{remark}

\numberwithin{equation}{section}

\newcommand{\PGL}{\mathop{\mathrm{PGL}}}

\newcommand{\SB}{\mathop{\mathrm{SB}}}

\newcommand{\CCor}{\mathop{\mathrm{cor}}}

\newcommand{\CM}{\mathop{\mathrm{CM}}}

\newcommand{\Spec}{\mathop{\mathrm{Spec}}}

\newcommand{\SO}{\mathop{\mathrm{SO}}}

\newcommand{\fonction}[5]{ #1~:\begin{array}{ccc}
 #2 & \longrightarrow & #3 \\
 #4 & \longmapsto & #5 \end{array}}

\begin{document}
\selectlanguage{english}
\begin{center}
{\Large
\textbf{Motivic equivalence of algebraic groups}

{\small Charles De Clercq}\\
}
\end{center}

\vspace{1cm}

{\small \textsc{Abstract.} Two semisimple algebraic groups of the same type are said to be motivic equivalent if the motives of the associated projective homogeneous varieties of the same type are isomorphic. We give general criteria of motivic equivalence in terms of the so-called higher Tits $p$-indices of algebraic groups. These results allow to give a complete classification of absolutely simple classical groups up to motivic equivalence in terms of the underlying algebraic structures. Among other applications of this classification, we deduce that there is a bijection between the stable birational equivalence classes of Severi-Brauer varieties of fixed dimension and the motivic equivalence classes of projective linear groups of the same rank.}

\vspace{1cm}

In this note we fix a prime $p$ and a base field $F$. By motive, we will always mean Grothendieck Chow motives with coefficients in $\mathbb{F}_p$, i.e. we will work in the category $\CM(F;\mathbb{F}_p)$ (see \cite{EKM}). Let $G$ and $G'$ be two semisimple algebraic groups which are twisted forms of the same algebraic group over the separable closure of the base field. We say that $G$ and $G'$ are $\ast$-equivalent if the $\ast$-action on the Dynkin diagram of $G$ and $G'$ coincide.

Recall that one can associate to any projective $G$-homogeneous variety $X$ a $\ast$-invariant subset $\Theta$ of the Dynkin diagram $\Delta_G$ of $G$. The isomorphism class of the projective $G$-homogeneous varieties of type $\Theta$ is denoted by $X_{\Theta,G}$.  We fix the convention that $X_{\Delta_G,G}$ is the isomorphism class of the variety of Borel subgroups of $G$, and we will often abuse notation and write $X_{\Theta,G}$ for a fixed projective $G$-homogeneous variety of type $\Theta$, and $F_{\Theta,G}$ for the function field of $X_{\Theta,G}$.

Consider a finite and separable field extension $L/F$ and a smooth projective $L$-variety $Y$. The corestriction $\CCor{}_{L/F}(Y)$ of $Y$ to $F$ is the $F$-variety $Y$ given by the composition $\Spec(L)\rightarrow \Spec(F)$ (see \cite{outer}). This functor on smooth projective  varieties induces the corestriction functor $\CCor{}_{L/F}:\CM(L;\mathbb{F}_p)\longrightarrow \CM(F;\mathbb{F}_p)$ on motives.

Let $G$ be a semisimple algebraic group and $E/F$ a minimal field extension such that $G_E$ is of inner type. If $\Theta$ is a subset of $\Delta_G$, consider a minimal field extension $L/F$ contained in $E$ such that $\Theta$ is $\ast$-invariant for $G_L$. We associate to the subset $\Theta$ the following data :
\begin{enumerate}
\item $M_{\Theta,G}$ is the motive $\CCor_{L/F}(X_{\Theta,G_{L}})$;
\item $U_{\Theta,G}$ is the upper motive of $M_{\Theta,G}$ (see \cite{upper}, \cite{outer}).
\end{enumerate}
The motive $M_{\Theta,G}$ (resp. $U_{\Theta,G}$) is the standard motive of type $\Theta$ (resp. standard upper motive of type $\Theta$) of $G$.\\

\begin{df}
Two semisimple and $\ast$-equivalent algebraic groups $G$ and $G'$ are \emph{motivic-equivalent modulo $p$} if for any subset $\Theta$ of the common Dynkin diagram of $G$ and $G'$, the standard motives of type $\Theta$ of $G$ and $G'$ are isomorphic. If $G$ and $G'$ are motivic equivalent modulo $p$ for any prime $p$, we will say that the algebraic groups $G$ and $G'$ are \emph{motivic equivalent}.
\end{df}

\section{Upper equivalence and tractable field extensions}

Let $X$ and $Y$ be two smooth projective varieties over $F$. The variety $Y$ dominates $X$ if there is a multiplicity $1$ correspondence $Y\rightsquigarrow X$ (with coefficients in $\mathbb{F}_p$). If moreover $X$ dominates $Y$, the varieties $X$ and $Y$ are said to be $p$-equivalent.

\begin{df}
Two semisimple and $\ast$-equivalent algebraic groups $G$ and $G'$ are \emph{upper-equivalent} if for any subset $\Theta$ of the common Dynkin diagram of $G$ and $G'$, the varieties $X_{\Theta,G}$ and $X_{\Theta,G'}$ are $p$-equivalent.
\end{df}

From now on, we will only consider semisimple $p$-inner algebraic groups, i.e. semisimple algebraic groups which become of inner type over a finite Galois field extension $E/F$ whose degree is a power of $p$. We say that $Y$ is a quasi-$G$-homogeneous variety if $Y$ is the corestriction to $F$ of a projective $G_L$-homogeneous variety for some intermediate field extension $E/L/F$. The upper motive of a projective quasi-$G$-homogeneous variety $Y$ is denoted by $U_Y$. According to \cite{upper}, two quasi-$G$-homogeneous varieties are $p$-equivalent if and only if their upper motives are isomorphic. In particular semisimple algebraic groups $G$ and $G'$ are upper equivalent if and only if for any subset $\Theta$, the standard upper motives $U_{\Theta,G}$ and $U_{\Theta,G'}$ are isomorphic.

\begin{df}
Let $G$ be a semisimple algebraic group and $X$ be a projective $G$-homogeneous variety. A field extension $L/F$ is \emph{tractable for $X$} if for any two quasi-$G$-homogeneous varieties $Y$ and $Y'$ which dominate $X$, $Y_L$ and $Y'_L$ are $p$-equivalent if and only if $Y$ and $Y'$ are $p$-equivalent.
\end{df}

The class of tractable field extensions for $X$ is obviously stable by sub-extensions and composition. Unirational field extensions are tractable for any projective homogeneous variety $X$. An essential observation is Lemma \ref{tractable}, which asserts that the function field of a projective homogeneous variety $X$ is tractable for $X$.

We will see in Proposition \ref{calc} that these extensions may be used to reconstruct the motive of a projective homogeneous variety over the base field. For the sake of formalizing this property, we introduce some material.

For any semisimple $p$-inner algebraic group $G$ of over $F$, let $\CM_G(F;\mathbb{F}_p)$ be the thick subcategory of $\CM(F;\mathbb{F}_p)$ generated by twisted direct summands of projective quasi-$G$-homogeneous varieties. The set of the upper motives of $G$, denoted by $U_G$, is the set of all the isomorphism classes of upper motives of projective quasi-$G$-homogeneous varieties. For any $M\in \CM_G(F;\mathbb{F}_p)$, the characteristic function of $M$ is defined as
$$\fonction{\Phi_M}{U_G\times \mathbb{Z}}{\mathbb{N}}{(U,i)}{\sharp\{\mbox{direct summands of $M$ isomorphic to $U[i]$}\}}.$$
Note that if $G$ and $G'$ are both $p$-inner and are upper equivalent, two motives $M$ and $M'$ which belong to $\CM_G(F;\mathbb{F}_p)$$=$$\CM_{G'}(F;\mathbb{F}_p)$ are isomorphic if and only if $\Phi_M=\Phi_{M'}$.

For any motive $M\in \CM_G(F;\mathbb{F}_p)$, we denote by $M^{\geq i}$ (resp. $M^{\leq i}$) the motive which is given by the direct sum of all the direct summands of $M$ which are upper motives shifted by some $j\geq i$ (resp. $j\leq i$). We also define $M^{<i}$ and $M^{>i}$ in the same way.

\begin{prop}\label{calc}
Assume $L/F$ is a tractable field extension for a projective $G$-homogeneous variety $X$ of type $\Theta$. For any projective quasi-$G$-homogeneous variety $Y$ dominating $X$,
$$\Phi_{M_{\Theta,G}}(U_Y,i)=\Phi_{(M_{\Theta,G_L})}(U_{Y_L},i)-\Phi_{(M_{\Theta,G}^{<i})_L}(U_{Y_L},i).$$
\end{prop}

\begin{proof}
The motivic decomposition $M_{\Theta,G}=M_{\Theta,G}^{\geq i}\oplus M_{\Theta,G}^{< i}$ yields
$$\Phi_{M_{\Theta,G_L}}(U_{Y_L},i)=\Phi_{(M_{\Theta,G}^{\geq i})_K}(U_{Y_L},i)+\Phi_{(M_{\Theta,G}^{< i})_K}(U_{Y_L},i),$$
we thus have to show that $\Phi_{M_{\Theta,G}}(U_Y,i)=\Phi_{(M_{\Theta,G}^{\geq i})_L}(U_{Y_L},i)$.

Consider a direct summand $M$ of $(M_{\Theta,G}^{\geq i})_L$ which is isomorphic to $U_{Y_L}[i]$. By the Krull-Schmidt property, the motive $M$ comes from an indecomposable summand $N$ of $M_{\Theta,G}^{\geq i}$. The theory of upper motives asserts that $N$ is isomorphic to some $U_{Y'}[j]$, where $Y'$ is a projective quasi-$G$-homogeneous variety which dominates $X$.

One easily sees that by the definition of $M_{\Theta,G}^{\geq i}$ , $i$ is equal to $j$ and the upper motive of $N_L$ must be isomorphic to $U_{Y_L}[i]$. However the upper motive of $N_L$ being isomorphic to $U_{Y'_L}$, we get an isomorphism between $U_{Y_L}$ and $U_{Y'_L}$. Since by assumption the field extension $L/F$ is tractable for $X$, the motives $U_{Y}$ and $U_{Y'}$ are isomorphic. We thus have shown that the number of direct summands isomorphic to $U_{Y}[i]$ in the motivic decomposition of $M_{\Theta,G}^{\geq i}$ is the same as the number of indecomposable summands isomorphic to $U_{Y_L}[i]$ inside $(M_{\Theta ,G}^{\geq i})_L$, that is to say $\Phi_{M_{\Theta,G}^{\geq i}}(U_{Y},i)=\Phi_{(M_{\Theta,G}^{\geq i})_L}(U_{Y_L},i)$. By the very definition of $M_{\Theta,G}^{\geq i}$, the left side of this equality is precisely $\Phi_{M_{\Theta,G}}(U_{Y},i)$.
\end{proof}

Note that if $G$, $G'$ are two upper-equivalent algebraic groups and $\Theta$ is a $\ast$-invariant subset for both $G$ and $G'$, a field extension $L/F$ is tractable for $X_{\Theta,G}$ if and only if it is tractable for $X_{\Theta,G'}$.

\begin{lem}\label{tractable}
Let $G$ be a semisimple algebraic group and let $X$ be a projective $G$-homogeneous variety. The function field $F(X)/F$ is tractable field extension for $X$.
\end{lem}

\begin{proof}We borrow the proof from \cite[Proof of Proposition 2.4]{suffgen}. Consider two projective quasi-$G$-homogeneous varieties $Y$ and $Y'$ which dominate $X$. The motives $U_{Y_{F(X)}}$ and $U_{Y'_{F(X)}}$ are isomorphic if and only if the $F(X)$-varieties $Y_{F(X)}$ and $Y'_{F(X)}$ are $p$-equivalent. Since $Y$ dominates $X$, the variety $X_{F(Y)}$ has a closed point whose residue field $L/F(Y)$ is of degree coprime to $p$. The function field $L(X)$ of the $L$-variety $X_L$ is a purely transcendental extension of $L$, since the projective $G_L$-homogeneous variety $X_L$ is isotropic.

Now by assumption $Y_{F(X)}$ dominates $Y'_{F(X)}$, thus the variety $Y'$ has a $0$-cycle of degree coprime to $p$ over the free composite of the function fields $F(Y)$ and $F(X)$, hence over its extension $L(X)$. The homotopy invariance of Chow groups yields a $0$-cycle of degree coprime to $p$ on $Y'_L$, hence over $F(Y)$ and $Y'$ dominates $Y$. Replacing $Y$ by $Y'$, we get another multiplicity $1$ correspondence $Y'\rightsquigarrow Y$, and thus $Y$ and $Y'$ are $p$-equivalent.
\end{proof}

\section{General criterion of motivic equivalence}

Let $G$ be a semisimple algebraic group and $\Theta$ be an $\ast$-invariant subset of the Dynkin diagram of $G$ such that the variety $X_{\Theta,G}$ is isotropic. In the same way as in \cite[Section 2]{chermergil}, we denote by $G_{\Theta}$ the semisimple part of a Levi subgroup of a parabolic subgroup of $G$ of type $\Theta$.

\begin{lem}\label{upgtheta}
If $G$ and $G'$ are upper equivalent and $\Theta$ is a subset of the common Dynkin diagram of $G$ and $G'$, then $G_{\Theta}$ and $G'_{\Theta}$ are upper-equivalent.
\end{lem}

\begin{proof}
The groups $G_{\Theta}$ and $G'_{\Theta}$ are isomorphic over $F_{sep}$ with the same Dynkin diagram $\Delta\setminus \Theta$. Consider a subset $\tilde{\Theta}$ of $\Delta \setminus \Theta$, and denote by $X$ (resp. $Y$) the variety $X_{\tilde{\Theta},G_{\Theta}}$ (resp. $X_{\tilde{\Theta},G'_{\Theta}}$). By assumption the varieties $X_{\tilde{\Theta},G}$ and $X_{\tilde{\Theta},G'}$ are $p$-equivalent. Since by \cite[Theorem 3.18]{kerstreh} the field extension $F_{\tilde{\Theta},G'}/F(Y)$ is purely transcendental, $Y$ dominates $X_{\tilde{\Theta},G}$. The variety $(X_{\tilde{\Theta},G})_{F(Y)}$ has thus a closed point whose residue field $L/F(Y)$ is a field extension of degree coprime to $p$.

The $L$-variety $X_{\tilde{\Theta},G_L}$ is projective homogeneous and isotropic, thus the field extension $L_{\tilde{\Theta},G_L}/L$ is purely transcendental. The projective $G_{\Theta}$-homogeneous variety $X$ has a $0$-cycle of degree coprime to $p$ over the function field $L_{\tilde{\Theta},G_L}$. The homotopy invariance of Chow groups yields a $0$-cycle of degree coprime to $p$ on $X_L$, hence $Y$ dominates $X$. Applying the same procedure exchanging $X$ and $Y$, the varieties $X$ and $Y$ are $p$-equivalent. Since the subset $\tilde{\Theta}$ of $\Delta \setminus \Theta$ is arbitrary, the algebraic groups $G_{\Theta}$ and $G'_{\Theta}$ are upper-equivalent.
\end{proof}

\begin{theo}\label{gencrit}Two semisimple and $p$-inner algebraic groups are motivic equivalent modulo $p$ if and only if they are upper-equivalent.
\end{theo}
\begin{proof}
The fact that two motivic equivalent algebraic groups $G$ and $G'$ modulo $p$ are upper motivic-equivalent follows from the Krull-Schmidt property of quasi-homogeneous varieties. We prove the converse for all semisimple algebraic groups over any field by induction on the common rank of $G$ and $G'$, i.e. the number of vertices of their common Dynkin diagram.

The base of the induction where the rank is $0$ is clear. Now assume that $rk(G)>0$ and consider a subset $\Theta$ of the common Dynkin diagram $\Delta$ of $G$ and $G'$. Excluding the trivial case where $\Theta$ is empty, we prove the result in three steps : we first prove the result in the case where $G$ and $G'$ both contain an $F$-defined parabolic subgroup of type $\Theta$. We then deal with the case where the subset $\Theta$ is $\ast$-invariant for both $G$ and $G'$, before treating the general case.

According to \cite[Theorem 7.5]{chermergil} (see also \cite{brosnan}), the motive of $X_{\Theta,G}$ decomposes into a direct sum of shifts of projective $G_{\Theta}$-varieties as long as $G$ contains parabolic subgroup of type $\Theta$. Moreover the shifts as well as the types of varieties appearing in this decomposition are determined by the data of the subset $\Theta$ and the $\ast$-action on the Dynkin diagram of $G$. The motive of $X_{\Theta,G'}$ also decomposes in the same way as a direct sum of shifts of projective $G'_{\Theta}$-homogeneous varieties, since $X_{\Theta,G'}$ is isotropic. The algebraic groups $G_{\Theta}$ and $G'_{\Theta}$ are upper equivalent by lemma \ref{upgtheta} and $rk(G_{\Theta})<rk(G)$. Those two decomposition coincide and thus the motives of $X_{\Theta,G}$ and $X_{\Theta,G'}$ are isomorphic by the induction hypothesis.

We now show how the result in the case where $\Theta$ is $\ast$-invariant for both $G$ and $G'$. First, observe that the free composite $L=F_{\Theta,G}.F_{\Theta,G'}$ of the function fields of the varieties $X_{\Theta,G}$ and $X_{\Theta,G'}$ is a tractable field extension for $X_{\Theta,G}$. Indeed, let $Y$ and $Y'$ be two quasi-$G$-homogeneous varieties which dominate $X_{\Theta,G}$, and such that $U_{Y_L}$ and $U_{Y'_L}$ are isomorphic. Since $G$ and $G'$ are upper equivalent, $Y$ and $Y'$ also dominate $X_{\Theta,G'}$, and through the external product of Chow groups the varieties $Y$ and $Y'$ both dominate the projective $G\times G'$-homogeneous variety $X_{\Theta,G}\times X_{\Theta,G'}$. Since the field extension $L$ corresponds to the function field of the $F$-variety $X_{\Theta,G}\times X_{\Theta,G'}$, the motives $U_Y$ and $U_{Y'}$ are isomorphic by Lemma \ref{tractable}.

Assume that the motives $M_{\Theta,G}$ and $M_{\Theta,G'}$ are not isomorphic. The characteristic function of those motives differ, thus we may choose the least integer $i_0$ such that the number of indecomposable summands of $M_{\Theta,G}$ and $M_{\Theta,G'}$ isomorphic to $U_Y$ differ, for some variety $Y$. The algebraic groups $G$ and $G'$ both contain a parabolic subgroup of type $\Theta$ over the free composite $F_{\Theta,G}.F_{\Theta,G'}$, thus by the previous discussion the motives of $X_{\Theta,G_L}$ and $X_{\Theta,G'_L}$ are isomorphic. It follows by Proposition \ref{calc} that
\begin{align}
\nonumber\Phi_{M_{\Theta,G}}(U_Y,i_0)&=\Phi_{(M_{\Theta,G})_L}(U_{Y_L},i_0)-\Phi_{(M_{\Theta,G}^{<i_0})_L}(U_{Y_L},i_0)\\
\nonumber&=\Phi_{(M_{\Theta,G'})_L}(U_{Y_L},i_0)-\Phi_{(M_{\Theta,G'}^{<i_0})_L}(U_{Y_L},i_0)\\
\nonumber&=\Phi_{M_{\Theta,G'}}(U_Y,i_0).
\end{align}
which contradicts the definition of $i_0$. In particular $\Phi_{M_{\Theta,G}}=\Phi_{M_{\Theta,G'}}$ and the motives of $X_{\Theta,G}$ and $X_{\Theta,G'}$ are isomorphic.

Finally if $\Theta$ is an arbitrary subset of the common Dynkin diagram of $G$ and $G'$, take a minimal field extension $K/F$ such that $\Theta$ is $\ast$-invariant for $G_K$. As $G$ and $G'$ are $\ast$-equivalent, $K$ is also a minimal field extension such that $\Theta$ is $\ast$-invariant for $G'_K$. The previous discussion shows that the motives of the $K$-varieties $M_{\Theta,G_K}$ and $M_{\Theta,G'_K}$ are isomorphic. Since $M_{\Theta,G}$ and $M_{\Theta,G'}$ are the images of those motives through the corestriction functor $\CCor_{K/L}$, the algebraic groups $G$ and $G'$ are motivic equivalent modulo $p$. 
\end{proof}

\section{Upper motives and the higher Tits indices}

The following result of Vishik shows that the motive of quadrics is essentially determined by the splitting behaviour of the associated quadratic form (see \cite[Criterion 0.1]{criteria} for a proof available in all characteristic different from $2$).

\begin{crit}[{\cite[Thm 4.18]{vish},\cite[Thm 1.4.1]{vish4}}] Let $q$ and $q'$ be quadratic forms over $F$ of the same dimension. The motives $M(Q)$ and $M(Q')$ with coefficients in $\mathbb{Z}$ of the associated quadrics are isomorphic if and only if for any field extension $E/F$, the Witt indices of $q_E$ and $q'_E$ coincide.
\end{crit}

We will now relate the notion of motivic equivalence and the splitting behaviour of algebraic groups. The set of the higher Tits indices are classical invariants which control the splitting of semisimple algebraic groups. For any semisimple algebraic $G$, the Tits index of $G$ is given by the data of the Dynkin diagram of $G$, a fixed subset of its vertices which is shaded, and the so-called $\ast$-action of the absolute Galois group (see \cite{tits}). A subset $\Theta$ is shaded in the Tits index of $G$ if and only if the type the parabolic subgroups of type $\Theta$ of $G$ are defined over the base field. Denoting by $Fields/F$ the category of the field extensions of $F$, the higher Tits index of a semisimple algebraic group $G$ is the functor $\mathrm{Tits}(G):Fields/F\longrightarrow Set$ which maps a field extension $E/F$ to the Tits index of $G_E$.

We define the $p$-local version of the Tits index as follows. The Tits $p$-index of $G$ consists of the data of the Dynkin diagram of $G$ endowed with the $\ast$-action of the absolute Galois group of $F$, and a fixed subset of its vertices which is shaded. A subset $\Theta$ is shaded in this Dynkin diagram if the variety $X_{\Theta,G}$ has a zero-cycle of degree coprime to $p$. The higher Tits $p$-index of $G$ is the functor $\mathrm{Tits}_p(G):Fields/F\longrightarrow Set$ which maps a field extension $E/F$ to the Tits $p$-index of $G_E$.

Even though at a first glance the fact that the motivic equivalence of algebraic groups is defined for motives with coefficients in $\mathbb{F}_p$ can be seen as a weakness, we will see that it is quite the opposite. The deep connections between motives and the higher Tits indices provided in the sequel don't hold for motives with coefficients in $\mathbb{Z}$.

In some sense these results show that motives with coefficients in $\mathbb{F}_p$ behave much better than motives with coefficients in $\mathbb{Z}$, notably in the study of the rational geometry of projective homogeneous varieties. An example of this difference is given by Karpenko's criterion of motivic equivalence of Severi-Brauer varieties \cite[Criterion 7.1]{criteria}, which can be seen as an evidence of the fact that the isomorphism classes of motives with coefficients in $\mathbb{Z}$ is not a birational invariant of Severi-Brauer varieties (compare with Corollary \ref{stabrat}).

\begin{df}
Let $p$ be a prime and $G$ be a semisimple algebraic group. We say that $G$ satisfies the property $(H_p)$ if $G$ is $p$-inner and if for any projective $G$-homogeneous variety $X$ and any field extension $E/F$, the variety $X_E$ is isotropic if and only if it has a $0$-cycle of degree coprime to $p$.
\end{df}

For instance, an absolutely simple adjoint algebraic group of inner type $A_n$ satisfies $(H_p)$ if and only if it is isomorphic to $\PGL(A)$, for some central simple algebra $A$ whose index is a power of $p$. If $q$ is a quadratic form defined over $F$, the special orthogonal group $\SO(q)$ of $q$ satisfies $(H_2)$ by Springer's theorem.

\begin{theo}\label{critpkern}
Two semisimple $p$-inner algebraic groups $G$ and $G'$ are motivic equivalent modulo $p$ if and only if the higher Tits $p$-indices of $G$ and $G'$ are equal.
\end{theo}

\begin{proof}
If the higher Tits indices coincide over any field extension, then a projective $G$-homogeneous variety of type $\Theta$ has a $0$-cycle of degree coprime to $p$ over the function field of a projective $G'$-homogeneous variety of the same type, and vice-versa. The varieties $X_{\Theta,G}$ and $X_{\Theta,G'}$ are thus $p$-equivalent, and in particular $G$ and $G'$ are upper equivalent, hence motivic equivalent modulo $p$ by Theorem \ref{gencrit}.

For the converse, fix an arbitrary field extension $E/F$ and denote $\mathrm{Tits}_p(G)(E)$ by $\Theta_0$ (resp. $\mathrm{Tits}_p(G')(E)$ by $\Theta_0'$). The variety $X_{\Theta_0,G}$ is isotropic over a residue field $L$ which is a coprime to $p$ field extension of $E$. Since $X_{\Theta_0,G'}$ has a $0$-cycle of degree coprime to $p$ over the function field $F_{\Theta_0,G}$, it is also the case over the function field $L_{\Theta_0,G}$, hence over $L$ since the field extension $L_{\Theta_0,G}/L$ is purely transcendental. The variety $X_{\Theta_0,G'}$ has thus a zero-cycle of degree coprime to $p$ over $E$. The subset $\Theta_0$ is thus shaded in the Tits $p$-index of $G'_E$, i.e. $\Theta_0$ is a subset of $\Theta_0'$. Exchanging the roles of $G$ and $G'$ yields the inclusion in the other direction, therefore the Tits $p$-indices of $G_E$ and $G'_E$ coincide.
\end{proof}

\begin{cor}\label{critpsp}
Assume that $G$ and $G'$ are $p$-inner and satisfy the property $(H_p)$. Then $G$ and $G'$ are motivic equivalent if and only if the higher Tits indices of $G$ and $G'$ are equal.
\end{cor}

We will provide the higher Tits (p)-indices for the different algebraic groups of classical type in order to give complete criteria for motivic equivalence of these groups in terms of the underlying algebraic structures.

\section{Motivic equivalence for algebraic groups of inner type $A_n$}

An absolutely simple adjoint algebraic group $G$ of inner type $A_n$ is isomorphic to the projective linear group $\PGL(A)$ of a central simple algebra $A$ of degree $n+1$. Let $\Theta=\{\alpha_{i_1},...,\alpha_{i_k}\}$ be a fixed subset of the Dynkin diagram of $G$. Any projective $G$-homogeneous variety of type $\Theta$ is isomorphic to the variety $X(i_1,...,i_k;A)$ of flags of right ideals of reduced dimension $i_1,...,i_k$ in $A$.

The description of the Tits index of $G$ is given in \cite{tits}. Fixing a prime $p$, we denote by $d_p$ the $p$-adic valuation of the index of $A$, the Tits $p$-index of $\PGL(A)$ is given by

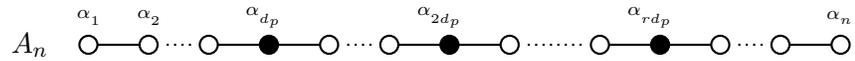
\begin{center}
  \begin{tikzpicture}[scale=.4]
    \draw (-1,0) node[anchor=east]  {$A_n$};
\draw (0.8,0.9) node[anchor=east]  {\tiny{$\alpha_{{}_{1}}$}};
\draw (2.8,0.9) node[anchor=east]  {\tiny{$\alpha_{{}_{2}}$}};
\draw (6.8,0.9) node[anchor=east]  {\tiny{$\alpha_{{}_{d_p}}$}};
\draw (12.8,0.9) node[anchor=east]  {\tiny{$\alpha_{{}_{2d_p}}$}};
\draw (19.8,0.9) node[anchor=east]  {\tiny{$\alpha_{{}_{rd_p}}$}};
\draw (25.8,0.9) node[anchor=east]  {\tiny{$\alpha_{{}_{n}}$}};
    \draw[xshift=0 cm,thick] (0 cm,0) circle (.3cm);
\draw[xshift=0.15 cm,thick] (0.15 cm,0) -- +(1.4 cm,0);
    \draw[xshift=1 cm,thick] (1 cm,0) circle (.3cm);
\draw[xshift=1 cm,dotted,thick] (1.6 cm,0) -- +(0.9 cm,0);
    \draw[xshift=2 cm,thick] (2 cm,0) circle (.3cm);
\draw[xshift=2.15 cm,thick] (2.15 cm,0) -- +(1.4 cm,0);
    \draw[xshift=3 cm,thick,fill=black] (3 cm,0) circle (.3cm);
\draw[xshift=3.15 cm,thick] (3.15 cm,0) -- +(1.4 cm,0);
    \draw[xshift=4 cm,thick] (4 cm,0) circle (.3cm);
\draw[xshift=4 cm,dotted,thick] (4.6 cm,0) -- +(0.9 cm,0);
    \draw[xshift=5 cm,thick] (5 cm,0) circle (.3cm);
\draw[xshift=5.15 cm,thick] (5.15 cm,0) -- +(1.4 cm,0);
    \draw[xshift=6 cm,thick,fill=black] (6 cm,0) circle (.3cm);
\draw[xshift=6.15 cm,thick] (6.15 cm,0) -- +(1.4 cm,0);
    \draw[xshift=7 cm,thick] (7 cm,0) circle (.3cm);
\draw[xshift=7 cm,dotted,thick] (7.6 cm,0) -- +(1.8 cm,0);
    \draw[xshift=8.5 cm,thick] (8.5 cm,0) circle (.3cm);
\draw[xshift=8.65 cm,thick] (8.65 cm,0) -- +(1.4 cm,0);
    \draw[xshift=9.5 cm,thick,fill=black] (9.5 cm,0) circle (.3cm);
\draw[xshift=9.65 cm,thick] (9.65 cm,0) -- +(1.4 cm,0);
    \draw[xshift=10.5 cm,thick] (10.5 cm,0) circle (.3cm);
\draw[xshift=10.5 cm,dotted,thick] (11.1 cm,0) -- +(0.9 cm,0);
    \draw[xshift=11.5 cm,thick] (11.5 cm,0) circle (.3cm);
\draw[xshift=11.65 cm,thick] (11.65 cm,0) -- +(1.4 cm,0);
    \draw[xshift=12.5 cm,thick] (12.5 cm,0) circle (.3cm);
  \end{tikzpicture}
\end{center}

\vspace{0,5cm}

Theorem \ref{critpkern} allow to give purely algebraic criterion for motivic equivalence of projective linear groups.

\begin{crit}[Inner type $A_n$ criterion]\label{critan}
Let $A$ and $B$ be two central simple $F$-algebras of the same degree. The projective linear groups $\PGL(A)$ and $\PGL(B)$ are motivic equivalent modulo $p$ if and only if the classes of the $p$-primary components of $A$ and $B$ generate the same subgroup in the Brauer group of $F$.
\end{crit}

\begin{proof}The functors $\mathrm{Tits}_p(\PGL(A))$ and $\mathrm{Tits}_p(\PGL(B))$ are equal if and only if for any field extension $E/F$, the index of the $p$-primary components of $A_E$ and $B_E$ coincide. This condition is well-known to be equivalent to the fact that these algebras generate the same subgroup in the Brauer group of $F$.
\end{proof}

As mentioned before, the above classification contrasts with Karpenko's \cite[Criterion 7.1]{criteria}. Assorted with the motivic dichotomy of projective linear groups, the following results hold.

\begin{cor}
Let $A$ and $B$ be two central simple algebras of the same degree and fix a sequence $i_1,...,i_k$ such that $X=X(i_1,...,i_k;A)$ and $Y=X(i_1,...,i_k;B)$ are anisotropic. The following conditions are equivalent :
\begin{enumerate}
\item the motives of $X$ and $Y$ are isomorphic with coefficients in $\mathbb{F}_p$;
\item the upper motives of $X$ and $Y$ are isomorphic;
\item the varieties $X$ and $Y$ are $p$-equivalent;
\item the projective linear groups $\PGL(A)$ and $\PGL(B)$ are motivic equivalent modulo $p$.
\end{enumerate}
\end{cor}

\begin{proof}
The upper motives of $X$ and $Y$ are both isomorphic to some upper motives $M$ and $M'$ of the projective linear groups associated to the division algebras $A_p$ and $B_p$ Brauer equivalent to the $p$-primary components of $A$ and $B$, respectively. If $X$ is anisotropic and the upper motives $U_X$ and $U_Y$ are isomorphic, then the projective linear groups of $\PGL(A_p)$ and $\PGL(B_p)$ are upper-equivalent, according \cite[Theorem 4.4]{dec3} (see also \cite{dec5}). Criterion \ref{critan} then implies that the projective linear groups associated to $A$ and $B$ are motivic equivalent modulo $p$, hence the motives of $X$ and $Y$ are isomorphic.
\end{proof}

These results show that the good motivic invariant to study the rational geometry of Severi-Brauer varieties is the collection of all the motives with $\mathbb{F}_p$ coefficients.

\begin{cor}\label{stabrat}Let $A$ and $B$ be two central simple algebras of the same degree. The Severi-Brauer varieties $\SB(A)$ and $\SB(B)$ are stably birationally equivalent if and only if the projective linear groups $\PGL(A)$ and $\PGL(B)$ are motivic equivalent.
\end{cor}

\begin{proof}
The Severi-Brauer varieties $\SB(A)$ and $\SB(B)$ are stably birationally equivalent if and only if for any field extension $E/F$ and any prime $p$, the index of the $p$-primary components of $A_E$ and $B_E$ are equal. This conditions is equivalent to the motivic equivalence modulo $p$ of $\PGL(A)$ and $\PGL(B)$, for any prime $p$.
\end{proof}

\section{Motivic equivalence of orthogonal groups}

If $q$ is a non-degenerate quadratic form over a field $F$ of characteristic different from $2$, the special orthogonal group $\SO(q)$ is either of type $B_n$ or $D_n$, whether the dimension of $q$ is odd or even. The Tits indices of orthogonal groups are provided in \cite{tits}. We higher Tits $p$-indices of orthogonal groups are as follows. If the prime $p$ is odd, the Tits $p$-index of $\SO(q)$ contains all the vertices of its Dynkin diagram. The only interesting prime is thus $p=2$, and $\SO(q)$ satisfies $(H_2)$ by Springer's theorem. The functors $\mathrm{Tits}(G)$ and $\mathrm{Tits}_2(G)$ are thus equal, and we now assume that $p=2$. 

If either $\dim(q)$ is odd or has trivial discriminant and even dimension, the orthogonal group of $q$ is of inner type and thus the motivic equivalence of such groups is determined by Theorem \ref{critpsp}. We thus focus in this section on the case of orthogonal groups of outer type $D_n$. 

\begin{lem}\label{witupper}
Let $q$ and $q'$ be two quadratic forms over $F$, such that $\dim(q)=\dim(q')$. If the orthogonal groups $\SO(q)$ and $\SO(q')$ are upper equivalent, then for any field extension $E/F$, the Witt indices of $q_E$ and $q'_E$ coincide.
\end{lem}

\begin{proof}Since the result corresponds to Proposition \ref{critpsp} for odd dimensional quadratic forms, we assume here that $\dim(q)=2k$. Assume that $\SO(q)$ and $\SO(q')$ are upper-equivalent, and fix a field extension $E/F$. We may exchange $q$ and $q'$ and thus assume that $i_W(q_E)=i$ is greater than $i_W(q'_E)$. We assume first that $i$ is at most $k-1$. In this case the variety $X(i;q_E)$ is isotropic, hence the field extension $E_{\{i\},\SO(q)}/E$ is purely transcendental. The upper motives $U_{\{i\},\SO(q)}$ and $U_{\{i\},\SO(q')}$ are isomorphic, thus the variety $X(i;q')$ has a rational point over $E_{\{i\},\SO(q)}$. The Witt index of $q'_E$ is thus greater than $i$, and $i_W(q_E)=i_W(q'_E)$. Now if $i_W(q_E)=k$, the orthogonal group $\SO(q_E)$ is split and its set of upper motives is reduced to the shifts of the Tate motive. Since $\SO(q)$ and $\SO(q')$ are upper equivalent, the motive of the quadric associated to $q'$ isomorphic to a direct sum of Tate motives, therefore $i_W(q'_E)=k$.
\end{proof}

\begin{crit}[Criterion for orthogonal groups]\label{critan}
Let $q$ and $q'$ be two quadratic forms of the same dimension. The orthogonal groups $\SO(q)$ and $\SO(q')$ are motivic equivalent if and only if for any field extension $E/F$, the Witt indices of $q_E$ and $q'_E$ coincide.
\end{crit}

\begin{proof}
The sufficient condition is given by lemma \ref{witupper}. Assuming that for any field extension $E/F$ the Witt indices of $q_E$ and $q'_E$ are equal, the associated orthogonal groups $\SO(q)$ and $\SO(q')$ are $\ast$-equivalent by \cite[Lemma 2.6]{criteria}, it thus remains to apply Theorem \ref{gencrit}.
\end{proof}

Note that even though the above criterion is stated in the context of motives with coefficients in $\mathbb{F}_2$, the result covers Vishik's criterion for the motivic equivalence of quadrics, since by \cite{vish4}, \cite[Theorem 1]{olivF2} (see also \cite[Théorème E.11.2]{kahnfq}), the change of coefficients functor induces a bijection between the set of isomorphism classes of motives of quadrics with coefficients in $\mathbb{Z}$ and $\mathbb{F}_2$. Assorted with Vishik's criterion, we get the following.

\begin{cor}
Let $q$ and $q'$ be two quadratic forms of the same dimension. The motives of the associated quadrics are isomorphic (with either coefficients in $\mathbb{Z}$ or $\mathbb{F}_2$) if and only if the orthogonal groups $\SO(q)$ and $\SO(q')$ are motivic equivalent.
\end{cor}

\section{Motivic equivalence and central simple algebras with involutions}

We now briefly describe the situation for other absolutely simple algebraic groups of classical type and inner type which are associated with central simple algebras with involutions. Again in the case where $p$ is an odd prime, all the vertices of the Dynkin diagram of such groups are shaded in their Tits $p$-index. We thus assume in the sequel that $p=2$ and we also assume that $char(F)$ is not $2$.

We briefly recall some notions on the invariants associated to central simple algebras with involutions. Assume that $(A,\sigma)$ is a central simple algebra with involution and $I$ is a right ideal in $A$. The orthogonal ideal $I^{\bot}$ of $I$ is defined as the annihilator of the left ideal $\sigma(I)$. An ideal $I$ is said to be isotropic if $I\subset I^{\bot}$. The index of the central simple algebras $(A,\sigma)$ with involution is the set of the reduced dimensions of the isotropic right ideals of $A$.

\begin{df}
Let $(A,\sigma)$ be a central simple algebra with involution. The $p$-index of $(A,\sigma)$ is the union of the indices of the $(A_L,\sigma_L)$, where $L$ goes through all the field extensions of $F$ of degree coprime to $p$.
\end{df}

\begin{crit}[Type $C_n$ criterion]Let $(A,\sigma)$ and $(B,\tau)$ be two central simple algebras of the same degree endowed with symplectic involutions of the first kind. The algebraic groups $Aut(A,\sigma)\!°$ and $Aut(B,\tau)\!°$ are motivic equivalent if and only if for any field extension $E/F$, the $2$-indices of $(A_E,\sigma_E)$ and $(B_E,\tau_E)$ are equal. 
\end{crit}

\begin{crit}[Inner type $D_n$ criterion]Let $(A,\sigma)$ and $(B,\tau)$ be two central simple algebras of the same degree endowed with orthogonal involutions. The algebraic groups $Aut(A,\sigma)\!°$ and $Aut(B,\tau)\!°$ are motivic equivalent if and only if for any field extension $E/F$, the $2$-indices of $(A_E,\sigma_E)$ and $(B_E,\tau_E)$ are equal. 
\end{crit}

\bibliographystyle{amsplain}

\end{document}